\documentclass[12pt,twoside, a4paper]{amsart}

\usepackage{mathptmx}
\usepackage{amssymb}
\usepackage{eucal}
\usepackage{amsxtra}
\usepackage{verbatim}
\usepackage{enumerate}
\usepackage[english]{babel}
\usepackage{bm}

\usepackage{setspace}
\usepackage[body={17cm,21cm}]{geometry}
\usepackage[T1]{fontenc}
\usepackage{graphicx}
\usepackage{mathrsfs}
\usepackage{amscd,latexsym,amsthm,amsfonts,amssymb,amsmath,amsxtra}
\usepackage[colorlinks, urlcolor=blue,  citecolor=blue]{hyperref}
\usepackage{enumerate}
\usepackage[all]{xy}%
\providecommand{\U}[1]{\protect\rule{.1in}{.1in}}
\RequirePackage{amsmath}
\RequirePackage{amssymb}
\usepackage{comment}

\theoremstyle{plain}

\newtheorem{theorem}{Theorem}
\newtheorem{proposition}[theorem]{Proposition}
\newtheorem{lemma}[theorem]{Lemma}
\newtheorem{corollary}[theorem]{Corollary}

\newtheorem{theorem?}{Theorem(?)} [section]
\newtheorem{proposition?}[theorem]{Proposition(?)}
\newtheorem{lemma?}[theorem]{Lemma(?)}
\newtheorem{corollary?}[theorem]{Corollary(?)}

\newtheorem*{theorem*}{Theorem}
\newtheorem*{proposition*}{Proposition}
\newtheorem*{lemma*}{Lemma}
\newtheorem*{corollary*}{Corollary}
\newtheorem*{question*}{Question}
\newtheorem*{conjecture*}{Conjecture}
\newtheorem*{claim*}{Claim}

\newtheorem*{introtheorem*}{Theorem}
\newtheorem*{introproposition*}{Proposition}
\newtheorem*{introlemma*}{Lemma}
\newtheorem*{introcorollary*}{Corollary}

\theoremstyle{definition}

\newtheorem*{definition*}{Definition}
\newtheorem*{example*}{Example}

\newtheorem{question}[theorem]{Question}

\theoremstyle{remark}
\newtheorem{remark}[theorem]{Remark}

\newtheorem*{remark*}{Remark}

\numberwithin{equation}{section}
\numberwithin{theorem}{section}

\DeclareSymbolFont{rsfs}{U}{rsfs}{m}{n}
\DeclareSymbolFontAlphabet{\mathcal}{rsfs}

\newcommand{\ZZ}{{\mathbb{Z}}}
\newcommand{\QQ}{{\mathbb{Q}}}

\newcommand{\Gal}{{\rm Gal}}
\newcommand{\Pic}{{\rm Pic}}

\newcommand{\Br}{{\rm Br}}

\newcommand{\Hom}{{\rm Hom}}

\newcommand{\inv}{{\rm inv}}

\newcommand{\Mor}{{\rm Mor}}

\newcommand{\R}{{\textup{R}}}

\def\Z{{\ZZ}}
\def\Q{{\QQ}}

\def\et{{\textup{\'et}}}

\def\fab{{\textup{f-ab}}}
\def\H{{\textup{H}}}

\def\A{\mathbf{A}}

\begin{document}

\title[]
{Remarks on Brauer--Manin obstruction for Weil restrictions}

\author{Sheng Chen}

\address{Sheng Chen \newline School of Mathematics and Statistics, \newline  Changchun University of Science and Technology; \newline 7089 Weixing Road,  Changchun , China}

\email{chenshen1991@cust.edu.cn}

\author{Kai Huang}

\address{Kai Huang \newline School of Mathematical Sciences, \newline  University of Science and Technology of China; \newline 96 Jinzhai Road, 230026 Hefei, China}

\email{hk0708@mail.ustc.edu.cn}
\date{\today}
\thanks{Corresponding author: Kai Huang; hk0708@mail.ustc.edu.cn; \newline School of Mathematical Sciences, University of Science and Technology of China; \newline 96 Jinzhai Road, 230026 Hefei, China}

\begin{abstract}
Given a finite extension $K/k$ of number fields and a smooth quasi-projective variety $X$ over $K$.
If $\pi_{1}^{\textup{ab}}(X\times_{K}\overline{k})$ is trivial, we prove that there is a natural identification between Brauer--Manin sets of $X$ and its Weil restriction $\R_{K/k}X$.
If $X$ is projective and $\Pic(X\times_{K}\overline{k})$ is a torsion-free abelian group, we prove that there is a natural identification between algebraic Brauer--Manin sets of $X$ and $\R_{K/k}X$.
\end{abstract}


\maketitle
\section{Introduction}
Let $k$ be a number field. We denote by $\Omega_k$ the set of places of $k$, by $k_v$ the completion of $k$ at $v$ and by ${\mathbf A}_k$ the ring of adeles of $k$.
Let $V$ be a smooth quasi-projective variety over $k$. The set $V(k)$ of $k$-rational points of $V$ can be viewed as a subset of the set of adelic points $V{({\mathbf A}_k)}$.

We say that $V$ satisfies Hasse principle if $V{({\mathbf A}_k)}\ne \emptyset $ implies $ V(k)\ne \emptyset$. In general Hasse principle fails for many varieties. After Manin \cite{Manin}, one may attempt to explain such failures by considering the Brauer--Manin set $V{({\mathbf A}_k)}^{\Br}$, defined as the set of elements of $V({\mathbf A}_k)$ that are orthogonal to the Brauer group $\Br(V)$ of $V$ with respect to the Brauer--Manin pairing. The global reciprocity law implies that $V(k)\subset V{({\mathbf A}_k)}^{\Br}$ (see \cite[\S 5.2]{Sko}). We say that there is a Brauer--Manin obstruction to Hasse principle for $V$ if $V({\mathbf A}_k)\ne \emptyset$ but $V{({\mathbf A}_k)}^{\Br}=\emptyset$.

Now for a finite extension $K$ of a number field $k$ and $X$ a quasi-projective variety over $K$, the {Weil restriction} $\R_{K/k}X$ exists and for any  $k$-scheme $Y$
$$\Mor_k(Y,\R_{K/k}X)=\Mor_K(Y\times_{k}K,X).$$
In particular, we have a natural identification
\begin{equation}\label{Iden}
\Phi: X(\A_K)\buildrel{=}\over\longrightarrow(\R_{K/k}X)(\A_k),
\end{equation}
under which the subset $X(K)$ is identified with $(\R_{K/k}X)(k)$.

In \cite[Remark 5, p. 95]{CTP}, J.-L. Colliot-Th\'el\`ene and B. Poonen ask whether the existences of Brauer--Manin obstruction to Hasse principle on $X$ and on $\R_{K/k}X$ are equivalent to each other. Furthermore, we can ask the following question.
\begin{question}\label{q}
Does the identification $\Phi$ preserve Brauer--Manin sets of $X$ and $\R_{K/k}X$?
\end{question}

By \cite[Corollary 3.2]{CL}, the inclusion $(\R_{K/k}X)(\mathbf{A}_{k})^{\Br} \subset \Phi(X(\mathbf{A}_{K})^{\Br})$ is known. However, the inclusion in the opposite direction remains open.

Similar questions for finite, finite solvable or finite abelian descent obstructions have
been affirmatively answered by M. Stoll in \cite[Proposition 5.15]{Sto}. Recently, Yang Cao and Yongqi Liang have also studied the \'etale Brauer--Manin set and obtained a positive answer (see [1, Theorem 1.1]).

In this note we give a positive answer to Question \ref{q} for varieties with trivial abelianized fundamental groups based on work of \cite{CL}. Our main results are the following.

\begin{theorem}\label{t0}(Theorem \ref{t1})
Let $X$ be a smooth quasi-projective variety over $K$. Assume that $\pi_{1}^{\textup{ab}}(X\times_{K}\overline{k})$ is trivial. Then $\Phi(X(\mathbf{A}_{K})^{\Br})=(\R_{K/k}X)(\mathbf{A}_{k})^{\Br}$.
\end{theorem}
\begin{theorem}(Theorem \ref{T})
If $X$ is a smooth projective variety over $K$ and $\Pic(X\times_{K}\overline{k})$ is a torsion-free abelian group. Then $\Phi(X(\mathbf{A}_{K})^{\Br_{1}})=(\R_{K/k}X)(\mathbf{A}_{k})^{\Br_{1}}$.
\end{theorem}

\section{Notation and preliminaries}
Let $k$ be a number field and $\overline{k}$  a fixed separable closure of $k$. We write $\Gamma_k=\Gal(\overline{k}/k)$. In this paper, a variety $V$ over $k$ means a geometrically integral and separated scheme of finite type over $k$. We denote by $k(V)$ the function field of $V$ and we write $\overline{V}$ for $V\times_{k}\overline{k}$.
For a variety $V$ over $k$, we recall that
$$\Br(V)=\H_{\text{\'et}}^2(V,\Bbb{G}_m),\ \ \Br_{0}(V)=\textup{Im}(\Br(k)\rightarrow \Br({V})),\ \
\Br_{1}(V)=\textup{Ker}(\Br(V)\rightarrow \Br(\overline{V})).$$

Let $\pi_{1}(V)$ be the \'etale fundamental group of $V$ and we denote by $\pi_{1}^{\textup{ab}}(V)$ the abelianization of $\pi_{1}(V)$, i.e., the quotient by the closure of the commutator subgroup.

For a variety $V$ over $k$, recall that
\begin{equation}
   V({\mathbf A}_k)^{\Br}=\{(P_v)_{v \in \Omega_k}\in V({\mathbf A}_k): \sum_{v \in  \Omega_k}\inv_v(\alpha(P_v))=0, \ \forall \alpha \in \Br(V)\}
\end{equation}
is the so-called Brauer--Manin set, and
\begin{equation}
   V({\mathbf A}_k)^{\Br_{1}}=\{(P_v)_{v \in \Omega_k}\in V({\mathbf A}_k): \sum_{v \in  \Omega_k}\inv_v(\alpha(P_v))=0, \ \forall \alpha \in \Br_{1}(V)\}
\end{equation}
is the so-called algebraic Brauer--Manin set.
Class field theory implies that $V(k)\subset V({\mathbf A}_k)^{\Br}\subset V({\mathbf A}_k)^{\Br_{1}}$.

\section{Varieties with trivial abelianized fundamental groups}
In this section we will give a proof of Theorem \ref{t0}.
\begin{theorem}\label{t1}
Let $K$ be a finite field extension of $k$ and $X$ a smooth quasi-projective variety over $K$. Assume that
$\pi_{1}^{\textup{ab}}(\overline{X})$ is trivial.
Then Question \ref{q} has a positive answer for $X$, i.e., $\Phi(X(\mathbf{A}_{K})^{\Br})=(\R_{K/k}X)(\mathbf{A}_{k})^{\Br}$.
\end{theorem}
\begin{proof}
Let $G$ be a finite abelian group scheme over $K$. Note that Hochschild--Serre spectral sequence gives the following exact sequence of cohomology groups
\begin{equation}\label{2.2}
    0 \to \H^{1}_{\et}(K,\H^{0}_{\et}(\overline{X}, \overline{G})) \to \H^{1}_{\et}(X,G)\to \H^{1}_{\et}(\overline{X}, \overline{G})^{\Gamma_K},
\end{equation}
where $\H^{1}_{\et}(\overline{X}, \overline{G})=0$ by assumption on  $\pi_{1}^{\textup{ab}}(\overline{X})$ and \cite[\S XI.5]{SGA1}. As $X$ is geometrically integral and $G$ is a finite group scheme over $K$, we have $\H^{0}_{\et}(\overline{X}, \overline{G})=G(\overline{K})$. Hence we obtain
\begin{equation}\label{2.3}
\H^{1}_{\et}(X,G)=\H^{1}_{\et}(K, G)
\end{equation}
by (\ref{2.2}).
Recall that $$X(\mathbf{A}_{K})^{\fab,\Br}=\bigcap\limits_{\substack{\text{$G$ finite abelian}\\ \text{all $G$-torsors }f\colon Y\to X}}\ \  \bigcup \limits_{[\sigma]\in \H^{1}_{\et}(K,G)} f_{\sigma}(Y^{\sigma}(\mathbf{A}_{K})^{\Br}).$$

For any $G$-torsor $f:Y\rightarrow X$, we see that there exists $[\tau]\in \H^{1}_{\et}(K,G)$ such that $Y^{\tau}$ is trivial by (\ref{2.3}). This implies that the map $f_{\tau} \colon Y^{\tau}\to X$ has a section. Thus we get that
$$f_{\tau}(Y^{\tau}(\mathbf{A}_{K})^{\Br})=X(\mathbf{A}_{K})^{\Br}.$$ So we have $$\bigcup\limits_{[\sigma]\in \H^{1}_{\et}(K,G)} f_{\sigma}(Y^{\sigma}(\mathbf{A}_{K})^{\Br})=X(\mathbf{A}_{K})^{\Br},$$ and then $X(\mathbf{A}_{K})^{\fab,\Br}=X(\mathbf{A}_{K})^{\Br}$.

On one hand, we see $(\R_{K/k}X)(\mathbf{A}_{k})^{\Br}\subset \Phi(X(\mathbf{A}_{K})^{\Br})$ by \cite[Corollary 3.2]{CL}.
On the other hand, we find $$\Phi(X(\mathbf{A}_{K})^{\Br})=\Phi(X(\mathbf{A}_{K})^{\fab,\Br})\subset (\R_{K/k}X)(\mathbf{A}_{k})^{\Br}$$
by \cite[Proposition 4.1]{CL}.
\end{proof}
\begin{corollary}
  Let $K$ be a finite field extension of $k$ and $X$ a smooth projective  variety over $K$ such that
$\pi_{1}^{\textup{ab}}(\overline{X})$ is trivial. Then the existences of Brauer--Manin obstruction to Hasse principle (resp. weak approximation) on $X$ and on $\R_{K/k}X$ are equivalent to each other.
\end{corollary}
\begin{proof}
It follows from the following commutative diagram
\begin{equation*}
    \xymatrix{
     \Phi(X(K))  \ar@{=}[r]\ar[d] &(\R_{K/k}X) (k)\ar[d]\\
     \Phi(X(\mathbf A_{K})^{\Br})    \ar@{=}[r]\ar[d]  &  (\R_{K/k}X)(\mathbf A_{k})^{\Br} \ar[d] \\
     \Phi(X(\mathbf A_{K}))  \ar@{=}[r] &  (\R_{K/k}X)(\mathbf A_{k}).
    }
\end{equation*}
\end{proof}

\begin{remark}
Note that for a smooth projective geometrically integral variety $X$ over a number field $k$,  $\pi_{1}^{\textup{ab}}(\overline{X})$ is trivial iff $\Pic(\overline{X})$ is torsion-free. This can be seen as follows. By Kummer sequence we have $\H^{1}_{\et}(\overline{X},\mu_{n})=\Pic(\overline{X})[n]$. Moreover, $\H^{1}_{\et}(\overline{X},\mu_{n})=\Hom_{cts}(\pi_{1}^{\textup{ab}}(\overline{X}),\Z/n\Z)$ by \cite[\S XI.5]{SGA1}. Hence $\Pic(\overline{X})[n]=\Hom_{cts}(\pi_{1}^{\textup{ab}}(\overline{X}),\Z/n\Z)$.  Thus we get $\Pic(\overline{X})_{\textup{tor}}=\Hom_{cts}(\pi_{1}^{\textup{ab}}(\overline{X}),\Q/\Z)=\Hom_{cts}(\pi_{1}^{\textup{ab}}(\overline{X}),\Bbb{R}/\Z) $ (note that  every continuous homomorphism $\phi: \pi_{1}^{\textup{ab}}(\overline{X}) \to \Bbb{R}/\Z$ has finite image). So we win by Pontryagin duality (cf. \cite[Theorem 1.1.11]{Neu}).
\end{remark}
\begin{proposition}\label{p1}
 Let $X$ and $Y$ be birationally equivalent smooth projective varieties over $K$. Assume that  $\Br(X)/\Br_{0}(X)$ is finite. Then Question \ref{q} has a positive answer for $X$ if and only if it has a positive answer for $Y$.
\end{proposition}
\begin{proof}
We denote briefly $V':=\R_{K/k}V$ for a variety $V$ over $K$. By \cite[Proposition 3.7.10]{CS}, we have
$$\Br(X)=\Br(Y), \ \ \Br(X')=\Br(Y').$$
It suffices to show
$$\Phi(X(\mathbf{A}_{K})^{\Br(X)})=X'(\mathbf{A}_{k})^{\Br(X')}\Longrightarrow \Phi(Y(\mathbf{A}_{K})^{\Br(Y)})=Y'(\mathbf{A}_{k})^{\Br(Y')}.$$
Note that there are open subsets $U\subset X$ and $E\subset Y$ such that $U\cong V$, hence we identify $U$ with an open subset of $Y$.  Since $\Phi(X(\mathbf{A}_{K})^{\Br(X)})=X'(\mathbf{A}_{k})^{\Br(X')}$, one concludes that
\begin{equation}\label{diag}
  U'(\mathbf{A}_{k})^{\Br(X')}=\Phi(U(\mathbf{A}_{K})^{\Br(X)})
\end{equation}
 by the following commutative diagram
\begin{equation*}
    \xymatrix{
     U'(\mathbf A_{k})  \ar@{=}[r]\ar[d] & \Phi(U(\mathbf A_K))\ar[d]\\
     X'(\mathbf A_{k})   \ar@{=}[r] & \Phi(X(\mathbf A_{K})).
    }
\end{equation*}
By (\ref{diag}) and \cite[Corollary 3.2]{CL}, we have  the following commutative diagram
\begin{equation*}
    \xymatrix{
     U'(\mathbf A_{k})^{\Br(Y')}  \ar@{=}[r]\ar[d] & \Phi(U(\mathbf A_K)^{\Br(Y)})\ar[d]\\
     Y'(\mathbf A_{k})^{\Br(Y')}    \ar[r] & \Phi(Y(\mathbf A_{K})^{\Br(Y)}).
    }
\end{equation*}
Since $\Br(Y)/\Br_0(Y)(=\Br(X)/\Br_0 (X))$ is finite, one concludes that $U(\mathbf A_{K})^{\Br(Y)}$ is dense in $Y(\mathbf A_K)^{\Br(Y)}$ by Harari's formal lemma
(see \cite[Th\'eor\`eme 1.4]{C}). Note that $Y'(\mathbf A_{k})^{\Br(Y')}$
is a closed subset in $Y(\mathbf A_{K})^{\Br(Y)}$  by \cite[Proposition 5.14]{Brian} and \cite[Proposition 13.3.1 (iv)]{CS},
hence $$\Phi(Y(\mathbf{A}_{K})^{\Br(Y)})=Y'(\mathbf{A}_{k})^{\Br(Y')}.$$
\end{proof}
\begin{remark}
One anonymous referee kindly pointed out to us that the finiteness of $\Br(X)/\Br_{0}(X)$ implies $\pi_{1}^{\textup{ab}}(\overline{X})=0$, and showed us the following arithmetic proof of this fact in the case $X(\mathbf{A}_{K})^{\Br(X)}\ne \emptyset$.

If $\pi_{1}^{\textup{ab}}(\overline{X})\ne 0$, we may assume $\H^{1}_{\et}(\overline{X}, \mu_n)\ne 0$. Since $X(\mathbf{A}_{K})^{\Br(X)}\ne \emptyset$, the universal torsor of $n$-torsion $f: W \to X$ exists by \cite[Corollary 3.6]{DS}. This is a torsor under a finite abelian group scheme $G$ over $K$ and we take an integral model $f: \mathcal{W}\to \mathcal{X}$ of this torsor. By \cite[Theorem 1.1]{CDX}, we have
$$X(\mathbf{A}_{K})^{\Br}\subset \bigcup \limits_{[\sigma]\in \H^{1}_{\et}(K,G)} f_{\sigma}(W^{\sigma}(\mathbf{A}_{K})).$$
Since $\Br(X)/\Br_0(X)$ is finite, the set $X(\mathbf{A}_{K})^{\Br}$ contains a compact open subset
$$\prod_{v \in S}Z(K_v)\times \prod_{v \notin S}\mathcal{X}(\mathcal{O}_v)$$ with $S$ finite. Note that the set
$$I=\{[\sigma] \in \H^{1}_{\et}(K,G):f_{\sigma}(W^{\sigma}(\mathbf{A}_{K}))\cap (\prod_{v \in S}Z(K_v)\times \prod_{v \notin S}\mathcal{X}(\mathcal{O}_v))\ne \emptyset\}$$
is finite by \cite[Proposition 6.3]{CDX}. Let $K_f$ be the algebraic closure of $K$ in the extension $K(W)^{gal}/K(X)$, where $K(W)^{gal}$ is the Galois closure of $K(W)/K(X)$. We take a finite field extension $L/K_f$ such that for any $\tau \in I$ the image of $[\tau]$ under the map $\H^{1}_{\et}(K,G)\to \H^{1}_{\et}(L,G)$ is trivial. Then $f: \mathcal{W}(\mathcal{O}_v)\to \mathcal{X}(\mathcal{O}_v)$ must be surjective for almost all $v\in J$, where $J=\{v \in \Omega_K: \text{v splits completely in L}\}$ (note that $[W]=[W^\tau]$ in $\H^{1}_{\et}(K_v,G)$ for $v\in J$). This is a contradiction by \cite[Proposition 3.5.2]{Se}.
\end{remark}
\begin{remark}
We do not know if Proposition \ref{p1} still holds in general when $\Br(X)/\Br_{0}(X)$ is infinite.
\end{remark}

\begin{remark}
By the purity theorem for Brauer group, one can easily check that if Question \ref{q} has a positive answer for $X$, then it also has a positive answer for open subsets of $X$ with complement of codimension $\ge 2$.
\end{remark}

\section{Comparison for Brauer--Manin sets under certain typical subgroups}
We refer to \cite[chap. VIII-X]{SGA} for definition and properties of algebraic groups of multiplicative type.
Recall for a number field $k$ with $\Gamma_k=\Gal(\overline{k}/k)$, there is a duality of categories between algebraic groups of multiplicative type over $k$ and discrete $\Gamma_k$-modules of finite type:
$$S\longmapsto \widehat{S}=\textup{Hom}_{\text{$\overline{k}$-gr}}(\overline{S},\mathbb{G}_{m,\overline{k}}).$$
Now let $K\subset \overline{k}$ be a finite field extension of $k$ with $\Gamma_K=\Gal(\overline{k}/K)$, $S$ an algebraic group of
multiplicative type over $K$.
\begin{lemma}\label{l}
There is a canonical isomorphism of $\mathbb{Z}[\Gamma_k]$-modules between $\widehat{\R_{K/k}S}$ and $\widehat{S}\otimes_{\mathbb{Z}[\Gamma_K]}\mathbb{Z}[\Gamma_k]$.
\end{lemma}
\begin{proof}This is a well-known lemma, we give its proof here for convenience of the reader. First we have the following isomorphism
$$\Hom_K(\R_{K/k}S\times_kK, S)\cong \Hom_k(\R_{K/k}S, \R_{K/k}S).$$
The identity on $\R_{K/k}S$ induces a $K$-morphism $\R_{K/k}S\times_kK \to S$, hence a homomorphism of $\mathbb{Z}[\Gamma_K]$-module $\widehat{S}\to \widehat{\R_{K/k}S}$. Then we get a homomorphism of $\mathbb{Z}[\Gamma_k]$-modules $\widehat{S}\otimes_{\mathbb{Z}[\Gamma_K]}\mathbb{Z}[\Gamma_k]\longrightarrow \widehat{\R_{K/k}S}$ in an obvious manner.

 Note that $\overline{\R_{K/k}S}$ is isomorphic to the $[K:k]$-multiple fiber product of $\overline{S}$  over $\overline{k}$, as algebraic groups over $\overline{k}$. One can easily check that the above homomorphism is bijective by forgetting the Galois group action.
\end{proof}

We refer to \cite[Section 2.3]{Sko} for definition and properties of torsors and their types. Now let $X$ be a smooth projective variety over $K$, and $S$ an algebraic
group of multiplicative type over $K$.    Note that there is a diagram with exact rows
\begin{equation}\label{COMM}
\xymatrix{
0\ar[r] & \H^{1}_{\et}(K,S) \ar[r]& \H^{1}_{\et}(X,S) \ar[r]& \Hom_{\mathbb{Z}[\Gamma_K]}(\widehat{S},\Pic(\overline{X}))
\\
0\ar[r] & \H^{1}_{\et}(k,\R_{K/k}S) \ar[r]& \H^{1}_{\et}(\R_{K/k}X,\R_{K/k}S) \ar[r]& \Hom_{\mathbb{Z}[\Gamma_k]}(\widehat{\R_{K/k}S},\Pic(\overline{\R_{K/k}X})),
\ar"1,2";"2,2"_{\cong}^{\R_{K/k}}
\ar"1,3";"2,3"^{\R_{K/k}}
\ar"1,4";"2,4"^{-\otimes_{\mathbb{Z}[\Gamma_K]}\mathbb{Z}[\Gamma_k]}
}
\end{equation}
where the left vertical map is an isomorphism by Shapiro's lemma. We will show that the diagram is commutative by the forthcoming Lemma \ref{commu}. Hence if $\lambda\in \Hom_{\mathbb{Z}[\Gamma_K]}(\widehat{S}, \Pic(\overline{X}))$ has image $\lambda'\in \Hom_{\mathbb{Z}[\Gamma_k]}(\widehat{\R_{K/k}S},\Pic(\overline{\R_{K/k}X}))$, the map $\H^{1}_{\et}(X,S)\rightarrow \H^{1}_{\et}(\R_{K/k}X,\R_{K/k}S)$ induces a bijection between torsors of type $\lambda$ and torsors of type $\lambda'$.

\begin{lemma}\label{commu}
The diagram (\ref{COMM}) is commutative.
\end{lemma}
\begin{proof}
\begin{enumerate}[(i)]
\item The left square commutes since Weil restriction is compatible with fiber products, i.e., $$\R_{K/k}(T\times_{K}X)=\R_{K/k}T\times_{k}\R_{K/k}X$$ for any $S$-torsor $T\rightarrow K$.
\item Let $q: \overline{\R_{K/k}S}\to \overline{S}$ and $p: \overline{\R_{K/k}X}\to \overline{X}$ be the canonical maps. By Lemma \ref{l} and its proof, the right square commutes
if and only if the following diagram commutes
\begin{equation}\label{cc}
\xymatrix@C=1ex{
 \small{\H^{1}_{\et}(X,S)\times \textup{Hom}_{\text{$\overline{k}$-gr}}(\overline{S},\mathbb{G}_{m,\overline{k}})}\ar[r]& \small{\H^{1}_{\et}(\overline{X},\overline{S})\times \textup{Hom}_{\text{$\overline{k}$-gr}}(\overline{S},\mathbb{G}_{m,\overline{k}})} \ar[r]& \small{\H^{1}_{\et}(\overline{X},\mathbb{G}_{m,\overline{k}})}
\\
\small{\H^{1}_{\et}(\R_{K/k}X,\R_{K/k}S) \times \textup{Hom}_{\text{$\overline{k}$-gr}}(\overline{\R_{K/k}S},\mathbb{G}_{m,\overline{k}})}\ar[r]&\small{ \H^{1}_{\et}(\overline{\R_{K/k}X},\overline{\R_{K/k}S})\times \textup{Hom}_{\text{$\overline{k}$-gr}}(\overline{\R_{K/k}S},\mathbb{G}_{m,\overline{k}}) }\ar[r]&\small{ \H^{1}_{\et}(\overline{\R_{K/k}X},\mathbb{G}_{m,\overline{k}})},
\ar"1,1";"2,1"_{(\R_{K/k}, \Tilde{q})}
\ar"1,3";"2,3"_{p^{*}}
}
\end{equation}
where the map $ \Tilde{q}:\textup{Hom}_{\text{$\overline{k}$-gr}}(\overline{S},\mathbb{G}_{m,\overline{k}})\rightarrow\textup{Hom}_{\text{$\overline{k}$-gr}}(\overline{\R_{K/k}S},\mathbb{G}_{m,\overline{k}})$ is induced by compositing with $q:\overline{\R_{K/k}S}\rightarrow \overline{S}$.

By diagram chasing and functoriality, the commutativity of (\ref{cc}) follows from the following commutative diagram
\begin{equation}
\xymatrix{
\H^{1}_{\et}(X,S)\ar[r]& \H^{1}_{\et}(\overline{X}, \overline{S}) \ar[r]^{p^{\ast}}& \H^{1}_{\et}(\overline{\R_{K/k}X}, \overline{S})\\
\H^{1}_{\et}(\R_{K/k}X,\R_{K/k}S)\ar[r]& \H^{1}_{\et}(\overline{\R_{K/k}X},\overline{\R_{K/k}S}),
\ar"1,1";"2,1"^{\R_{K/k}}
\ar"2,2";"1,3"^{q_{*}}
}
\end{equation}
which can be obtained by straightforward computation of torsors.
\end{enumerate}
\end{proof}
\begin{corollary}\label{cor}
Let $X$ be a smooth projective variety over $K$, also assume that $\Pic(\overline{X})$ is a torsion-free abelian group. Let  $T$ be the dual torus of  $\Pic(\overline{X})$.
Then $\Pic(\overline{\R_{K/k}X})$ is a $\Gamma_k$-module of finite type, and its dual is just $\R_{K/k}T$. Moreover, the map $\H^{1}_{\et}(X,T)\rightarrow \H^{1}_{\et}(\R_{K/k}X,\R_{K/k}T)$ induces a bijection between universal torsors of $X$ and $\R_{K/k}X$.
\end{corollary}
\begin{proof}
According to \cite[Definition 2.3.3]{Sko} a $T$-torsor on $X$ is called a universal torsor if its type is the identity in $\textup{Hom}_{\mathbb{Z}[\Gamma_K]}(\widehat{T},\Pic(\overline{X}))$.
By \cite[Proposition 1.7]{Sk}, we can obtain the following isomorphism of $\mathbb{Z}[\Gamma_k]$-module
$$\Pic(\overline{\R_{K/k}X})\cong \Pic(\overline{X})\otimes_{\mathbb{Z}[\Gamma_K]}\mathbb{Z}[\Gamma_k],$$
then the first assertion comes from Lemma \ref{l}. And the second assertion follows from the commutative diagram above by taking $S=T$.
\end{proof}
Let $r: \Br_1(X)\to \H^{1}_{\et}(K,\Pic(\overline{X}))$
be the canonical map from the Hochschild--Serre spectral sequence $\H^{p}_{\et}(K, \H^{q}(\overline{X}, \Bbb{G}_m))\Rightarrow \H^{p+q}_{\et}(X, \Bbb{G}_m)$ (see \cite[Corollary 2.3.9]{Sko}). Let $M$ be a $\Gamma_K$-module of finite type, and $\lambda: M \to \Pic(\overline{X})$ a homomorphism of $\Gamma_K$-modules. Recall that
$$\Br_{\lambda}(X):=r^{-1}\lambda_{\ast}(\H^{1}_{\et}(K, M))\subset \Br_{1}(X).$$
We define $X(\mathbf{A}_{K})^{\Br_{\lambda}}\subset X(\mathbf{A}_{K})$ as the set of adelic points orthogonal to $\Br_{\lambda}$ with respect to the Brauer--Manin pairing. The main result of this section is the following.
\begin{theorem}\label{T}
Let $X$ be a smooth projective variety over $K$, and $S$ an algebraic group of multiplicative
type over $K$. If $\lambda\in \Hom_{\mathbb{Z}[\Gamma_K]}(\widehat{S}, \Pic(\overline{X}))$ has image $\lambda'\in \Hom_{\mathbb{Z}[\Gamma_k]}(\widehat{\R_{K/k}S},\Pic(\overline{\R_{K/k}X}))$, then $\Phi(X(\mathbf{A}_{K})^{\Br_{\lambda}})=(\R_{K/k}X)(\mathbf{A}_{k})^{\Br_{\lambda'}}$.

Moreover if  $\Pic(\overline{X})$ is a torsion-free abelian group, then $\Phi(X(\mathbf{A}_{K})^{\Br_{1}})=(\R_{K/k}X)(\mathbf{A}_{k})^{\Br_{1}}$.
\end{theorem}
\begin{proof}
We denote briefly $V':=\R_{K/k}V$ for a variety $V$ over $K$.
Given an $S$-torsor $f:Y\rightarrow X$ of type $\lambda$, by assumption its Weil restriction $f':Y'\rightarrow X'$ is an $S'$-torsor of type $\lambda'$, and $f(Y(\mathbf{A}_{K}))= f'(Y'(\mathbf{A}_{k}))$ under the  natural identification $\Phi: X(\A_K)\buildrel{=}\over\longrightarrow X'(\A_k)$. Now by Corollary \ref{cor} and the discussion before it, Weil restriction induces a bijection between torsors of type $\lambda$ and torsors of type $\lambda'$. So we have $$\Phi(X(\mathbf{A}_{K})^{\Br_{\lambda}})=\Phi(\bigcup \limits_{\text{type $(Y,f)=\lambda$}} f(Y(\mathbf{A}_{K})))=\bigcup \limits_{\text{type $(Y',f')=\lambda'$}} f'(Y'(\mathbf{A}_{k}))=(\R_{K/k}X)(\mathbf{A}_{k})^{\Br_{\lambda'}},$$
$$\Phi(X(\mathbf{A}_{K})^{\Br_{1}})=\Phi(\bigcup \limits_{\text{type $(Y,f)=\text{Id}$}} f(Y(\mathbf{A}_{K})))=\bigcup \limits_{\text{type $(Y',f')=\text{Id}$}} f'(Y'(\mathbf{A}_{k}))=(\R_{K/k}X)(\mathbf{A}_{k})^{\Br_{1}}$$
by \cite[Theorem 6.1.2 (a)]{Sko}.
\end{proof}
\bf{Acknowledgments}	
\it{We thank Yang Cao and Yongqi Liang for many useful discussions and suggestions. We also thank the referees for their careful scrutiny and valuable comments. This work is supported by National Natural Science Foundation of China No.12071448}.

 \end{document}